\documentclass{amsart}

\usepackage{amssymb}
\usepackage{amsfonts}
\usepackage{amsmath}
\usepackage{amsthm}
\usepackage{graphicx}
\usepackage{mathrsfs}
\usepackage{dsfont}
\usepackage{amscd}
\usepackage[shortlabels]{enumitem}
\normalfont
\usepackage[T1]{fontenc}
\usepackage{calligra}
\usepackage{verbatim}
\usepackage[usenames,dvipsnames]{color}
\usepackage[colorlinks=true,linkcolor=Blue,citecolor=Violet]{hyperref}

\makeatletter
\@namedef{subjclassname@2010}{%
  \textup{2010} Mathematics Subject Classification}
\makeatother

\theoremstyle{plain}
\newtheorem{theorem}{Theorem}[section]

\newtheorem{case}{Case}

\newtheorem{corollary}[theorem]{Corollary}

\newtheorem{lemma}[theorem]{Lemma}
\newtheorem{proposition}[theorem]{Proposition}

\newtheorem{theorem-definition}[theorem]{Theorem-Definition}

\theoremstyle{definition}
\newtheorem{definition}[theorem]{Definition}

\newtheorem{notation}[theorem]{Notation}

\newtheorem{convention}[theorem]{Convention}

\theoremstyle{remark}

\newtheorem{example}[theorem]{Example}

\numberwithin{equation}{section}

\newcommand{\lac}{{\mathsf{L}^\lambda_{\mathcal{T}_{\mathcal{C}}}}}
\newcommand{\ldc}{{\mathsf{L}_{\mathsf{d}_{\mathcal{C}}}}}
\newcommand{\dc}{{\mathsf{d}_{\mathcal{C}}}}
\newcommand{\dd}{{\mathsf{d}}}

\newcommand{\N}{{\mathds{N}}}

\newcommand{\R}{{\mathds{R}}}
\newcommand{\C}{{\mathds{C}}}

\newcommand{\D}{{\mathfrak{D}}}
\newcommand{\A}{{\mathfrak{A}}}
\newcommand{\B}{{\mathfrak{B}}}

\newcommand{\Lip}{{\mathsf{L}}}

\newcommand{\Kantorovich}[1]{{\mathsf{mk}_{#1}}}

\newcommand{\mongekant}{{Mon\-ge-Kan\-to\-ro\-vich metric}}

\newcommand{\sa}[1]{{\mathfrak{sa}\left({#1}\right)}}

\newcommand{\dom}[1]{{\operatorname*{dom}({#1})}}

\renewcommand{\geq}{\geqslant}
\renewcommand{\leq}{\leqslant}

\newcommand{\Latremoliere}{Latr\'{e}moli\`{e}re}

\allowdisplaybreaks[1]

\makeatletter
\newcommand{\vast}{\bBigg@{4}}
\newcommand{\Vast}{\bBigg@{5}}
\makeatother

\begin{document}

\title[A quantum metric on the Cantor space]{A quantum metric on the Cantor Space}\author{Konrad Aguilar}
\address{School of Mathematical and Statistical Sciences, Arizona State University, 901 S. Palm Walk, Tempe, AZ 85287-1804}
\email{konrad.aguilar@asu.edu}
\urladdr{https://math.la.asu.edu/~kaguilar/}

\author{Alejandra L\'opez}
\address{Department of Mathematics, Purdue University,
150 N. University Street, West Lafayette, IN 47907-2067}
\email{lopez408@purdue.edu}
\urladdr{}

\date{\today}

\begin{abstract}
The first author and \Latremoliere{} had introduced a quantum metric (in the sense of Rieffel) on the algebra of complex-valued continuous functions on the Cantor space. We show that this quantum metric is distinct from the quantum metric induced by a classical metric on the Cantor space. We accomplish this by showing that the seminorms  induced by each quantum metric (Lip-norms) are distinct on a dense subalgebra of the algebra of complex-valued continuous functions on the Cantor space. In the process, we develop formulas for each Lip-norm on this dense subalgebra and show these Lip-norms agree on  a Hamel basis of this subalgebra. Then, we use these formulas to find families of elements for which these Lip-norms disagree.
\end{abstract}
\maketitle

\setcounter{tocdepth}{3}
\tableofcontents

\section{Introduction and Background}

The study of compact quantum metric spaces was introduced by Rieffel \cite{Rieffel98a, Rieffel05} to establish metric convergence of certain noncommutative algebras. This metric convergence also serves as noncommutative analogue to the Gromov-Hausdorff distance \cite{burago01} which provides metric convergence of sequences {\em of} compact metric spaces \cite{Rieffel00}. This was motivated by a desire to formalize convergence of certain noncommutative algebras introduced in the high-energy physics literature \cite{Rieffel00}. Another aspect of this theory produced a way to study finite-dimensional approximations of infinite dimensional algebras using this strong form of metric convergence. Therefore, although this theory was developed for noncommutative algebras, the pursuit of metric finite-dimensional approximations meant that this theory could be of interest to study finite-dimensional approximations of commutative algebras. 

Producing metric finite-dimensional approximations for some commutative algebras was one of the consequences of the work of the first author and \Latremoliere{} in \cite{Aguilar-Latremoliere15}. The commutative algebra they focused on was the algebra of complex-valued continuous functions on the Cantor space (where the Cantor space is viewed as sequences of $0$'s and $1$'s), denoted $C(\mathcal{C})$. They accomplished these finite-dimensional approximations by placing a quantum metric on $C(\mathcal{C})$ using the group structure of the Cantor space since the Cantor space is a compact group \cite[Theorem 3.5 and Section 7]{Aguilar-Latremoliere15}. However, the Cantor space is also a compact metric space \cite[Theorem 30.5]{Willard}, and therefore has a classical quantum metric on it induced by the Lipschitz constant associated to the metric on the Cantor space. Now, each of these quantum metrics is induced by a seminorm, called a Lip-norm (Lip is short for Lipschitz), on $C(\mathcal{C})$. So, the natural question is whether these Lip-norms are the same. It is not too difficult to place distinct Lip-norms on a given commutative or noncommutative space, but it was also shown in \cite{Aguilar-Latremoliere15} that these two quantum metrics are strongly  related to each other in \cite[Corollary 7.6]{Aguilar-Latremoliere15} (we also state this relation in Theorem \ref{t:lac}). This relation is strong enough to provide an equivalence to when these Lip-norms are same on a dense subspace \cite[Theorem 8.1]{Rieffel99}.  Thus, it is not entirely trivial to establish a difference in these Lip-norms, which is a main accomplishment of this paper. Hence, our work suggests that it was important for \cite{Aguilar-Latremoliere15} to introduce this new quantum metric to achieve their finite-dimensional approximations.   We show that these Lip-norms are distinct by   introducing formulas for them on an infinite-dimensional  dense subalgebra of $C(\mathcal{C}).$ We also show that these Lip-norms agree on a Hamel basis for this dense subalgebra, which makes it even more surprising when we do find a family of elements where they disagree.

For the rest of this section, we provide sufficient background for the other two sections. In Section \ref{s:formulas}, we show that these Lip-norms agree on a Hamel basis for a dense subalgebra (Theorem \ref{t:basis-compare}) and also provide formulas for both Lip-norms built from the structure of this dense subalgebra (Theorem \ref{t:ldc2-formula} and Theorem \ref{t:lac-formula}). In Section \ref{s:separating}, using the results from the previous section, we provide more explicit formulas for these Lip-norms on a certain finite-dimensional subspace (Theorem \ref{t:ldc2-formula} and Theorem \ref{t:lac2-formula}) and use these formulas to separate these Lip-norms. Finally, we find a comparison of these Lip-norms on this finite-dimensional subspace. 

Now, we begin the background. We start with necessary definitions to define a compact quantum metric space. A compact quantum metric space is a certain kind of algebra called a C*-algebra with a special type of seminorm defined on it.  Thus, we define algebras now. Definitions (\ref{algebra-def}---\ref{states-def}) are contained in \cite[Chapter I]{Davidson}.
 \begin{definition}\label{algebra-def}
 An {\em associative algebra over the complex numbers} $\C$ is a vector space $A$ over $\C$ with an associative multiplication, denoted by concatenation, such that:
 \begin{align*}
 & a(b+c)=ab+ac \text{ and } (b+c)a=ba+ca \text{ for all } a,b,c \in A \\
 & \lambda (ab)=(\lambda a) b=a(\lambda b)  \text{ for all } a,b \in A, \lambda \in \C.
 \end{align*}
 In other words, the associative multiplication is a bilinear map from $A \times A$ to $A$. We denote the zero of a vector space by $0_A.$
 
 We say that $A$ is {\em unital} if there exists a multiplicative identity, denoted by $1_A$.  That is:
 \begin{equation*}
 1_A a=a=a1_A \text{ for all } a \in A.
 \end{equation*}
 \end{definition}
 \begin{convention}
All algebras are associative algebras over the complex number $\C$.
\end{convention}
\begin{notation}
When $E$ is a normed vector space, then its norm will be denoted by $\|\cdot\|_E$ by default.
\end{notation}
\begin{definition}\label{banach-algebra}
A {\em normed algebra} is an algebra $A$ with a norm $\Vert \cdot \Vert_A$ such that:
\begin{equation*}
\Vert ab \Vert_A \leq \Vert a \Vert_A \Vert b \Vert_A \text{ for all } a,b \in A.
\end{equation*}

$A$ is a  {\em Banach Algebra}  when $A$ is complete with respect to the norm $\Vert \cdot \Vert_A$.
\end{definition}
\begin{definition}\label{c*-algebra}
A {\em C*-algebra}, $A$, is a Banach algebra such that  there exists an anti-multiplicative conjugate linear involution ${}^* : A \longrightarrow A$, called the {\em adjoint}.  That is, * satisfies:
\begin{enumerate}
\item  (conjugate linear): $(\lambda (a+b))^* = \overline{\lambda}(a^* + b^*) \text{ for all } \lambda \in \C, a,b \in A;$  
\item  (involution):  $(a^*)^*=a \text{ for all } a \in A;$ 
\item  (anti-multiplicative):   $(ab)^* = b^*a^* \text{ for all } a,b \in A.$ 
\end{enumerate}
Furthermore, the norm, multiplication, and  adjoint together satisfy the identity:
 \begin{equation}\Vert a a^*\Vert_A= \Vert a \Vert_A^2 \text{ for all } a \in A
\end{equation}
 called the {\em C*-identity}.
 \begin{comment}
 The set of self-adjoint elements of a C*-algebra is the set $\sa{\A}=\{a \in \A : a=a^*\}$.

An element $u \in \A$ of a unital C*-algebra is {\em unitary} if $uu^*=u^*u=1_\A$.
\end{comment}

We say that $B \subseteq A$ is a {\em C*-subalgebra of} $A$ if $B$ is a norm closed subalgebra that is also self-adjoint, i.e. $a \in B \iff a^* \in B$.  
\end{definition}

Our main example will be the C*-algebra of complex-valued continuous functions on a compact metric space, which we define now.

\begin{example}[{\cite[Example I.1.2]{Davidson}}]\label{e:cont-fcn}
Let $(X, \mathsf{d}_X)$ be a compact metric space. Define
\[
C(X)=\{ f: X \rightarrow \C \mid f \text{ is continuous}\}
\]
This is a C*-algebra under pointwise algebraic operations including pointwise complex conjugation as the involution. That is, if $f \in C(X),$ then $f^*=\overline{f}$, which is defined for all $x \in X$, by 
\[
\overline{f}(x)=\overline{f(x)}.
\]
The C*-norm is given for all $f \in C(X)$ by \[\|f\|_{C(X)}=\|f\|_\infty=\sup_{x \in X} \{|f(x)|\}.\]
The unit is the constant $1$ function denoted $1_{C(X)}$ defind for all $x \in X$ by 
\[
1_{C(X)}(x)=1.
\]
\end{example}

In order to define  compact quantum metric spaces we need to define another structure associated to C*-algebras.
\begin{definition}\label{states-def}
Let $A$ be a unital C*-algebra. Let $A'$ denote the set of continuous and linear complex-valued functions on $A$. The {\em state space} of $A$ is the set
\begin{equation*} S(A)=\left\{ \varphi \in A' \mid  1=\varphi(1_A) = \Vert \varphi \Vert_{A'} \right\},
\end{equation*}
where $\Vert \varphi \Vert_{A'} = \sup \{\vert \varphi(a) \vert : a \in A, \Vert a \Vert_A =1\}$ is the operator norm. 
\end{definition}
We also need the notion of a seminorm, which we will allow to take value $\infty$ in this article.
\begin{definition}
Let $V$ be a vector space over $\C$.  A  {\em seminorm} $s$ on $A$ is a function
\[
s:A \rightarrow [0,\infty]
\]
such that
\begin{enumerate}
    \item $s(0_V)=0$,
    \item (homogeneity) $s(\lambda a)=|\lambda|s(a)$ for all $\lambda \in \C, a \in A,$
    \item (triangle inequality) $s(a+b)\leq s(a)+s(b)$ for all $a,b \in A.$
\end{enumerate}
\end{definition}
Now, we define  compact quantum metric spaces.

\begin{definition}[\cite{Rieffel98a,Rieffel99, Rieffel05,Latremoliere15}]\label{d:cqms}
A {\em compact quantum  metric space} $(A,\Lip)$ is an ordered pair where $A$ is a unital C*-algebra with unit $1_A$  and $\Lip$ is a seminorm on $A$ such that  $\dom{\Lip}=\{a \in A\mid \Lip(a)< \infty\}$ is dense in $A$, and:
\begin{enumerate}
\item $\Lip(a)=\Lip(a^*)$ for all $a \in A,$
\item $\{ a \in A \mid \Lip(a) = 0 \} = \C1_A$,
\item the seminorm $\Lip$ is lower semi-continuous on $A$ with respect to $\|\cdot\|_A$, and
\item the \emph{\mongekant} defined, for all two states $\varphi, \psi \in S(A)$, by
\begin{equation*}
\Kantorovich{\Lip} (\varphi, \psi) = \sup\left\{ |\varphi(a) - \psi(a)| \mid a\in A, \Lip(a) \leq 1 \right\}
\end{equation*}
is a metric on $S(A)$ such that there exists $D> 0$ such that
\begin{enumerate}
    \item $\mathrm{diam}(S(A), \Kantorovich{\Lip})\leq D$, and
    \item the set $\{a \in A \mid \Lip(a) \leq 1, \|a\|_A \leq D\}$ is compact in $A$.
\end{enumerate} 
\end{enumerate}
If $(A,\Lip)$ is a  compact quantum metric space, then we call the seminorm $\Lip$ a {\em Lip-norm}.  

Furthermore, if there exists $C\geq 1$ such that 
\[
\Lip(ab) \leq C (\Lip(a) \|b\|_A+ \Lip(b)\|a\|_A)
\]
 
for all $a,b \in A$, then we call $\Lip$, {\em $C$-quasi-Leibniz} and we call $(A, \Lip)$ a { \em $C$-quasi-Leibniz compact quantum metric space.} If $C=1$, then We call $\Lip$, {\em Leibniz}.
\end{definition}

One of the main examples of a quantum metric is the following and is due to Kantorovich, although Kantorovich did not call such an object a quantum metric. First, some definitions, the first of which allows us to show that this quantum metric of Kantorovich recaptures the classical metric.
\begin{definition}\label{t:Dirac-pm}
Let $(X, \mathsf{d}_X)$ is a compact metric space. Let $x \in X.$  We define the {\em Dirac point mass at $x$} to be the function
\[
\delta_x : f \in C(X) \mapsto f(x) \in \C,
\]
where $\delta_x \in S(C(X))$ \cite[Theorem VII.8.7]{Conway90}.
\end{definition}
We note that the set of all Dirac point masses on $C(X)$ is the set of certain kinds of states called {\em pure states} (see \cite[Theorem VII.8.7]{Conway90} and \cite[Theorem 5.1.6]{Murphy90}),  but we do not need to study this fact deeper in this article. Now, we define a main Lip-norm for this paper, which shows that quantum metrics can recover classical metrics.
\begin{definition}\label{d:Lipschitz-constant}
Let $(X, \mathsf{d}_X)$ be a compact metric space. The {\em  Lipschitz seminorm} on $C(X)$ is defined for all $f\in C(X)$ by
\[
\Lip_{\mathsf{d}_X}(f)=\sup_{x,y \in X, x \neq y} \left\{\frac{|f(x)-f(y)|}{\mathsf{d}_X(x,y)}\right\}.
\]
\end{definition}

\begin{theorem}[{\cite{Kantorovich40,Rieffel98a}}]\label{t:classical-qm}
If $(X, \mathsf{d}_X)$ be a compact metric space, then $(C(X), \Lip_{\mathsf{d}_X})$ is a Leibniz compact quantum metric space.

Moreover, for all $x,y \in X,$ it holds that
\[
\mathsf{d}_X(x,y)=\Kantorovich{\Lip_{\mathsf{d}_X}}(\delta_x, \delta_y).
\]
\end{theorem}
In this paper we will consider the particular compact metric space given by the Cantor space, which we define now.
\begin{convention}
The natural numbers $\N$ contain $0$ throughout this article.
\end{convention}
\begin{definition}[{\cite[Corollary  30.5]{Willard}}]\label{d:Cantor-space}
The {\em Cantor space} is the set of sequences of $0$'s and $1$'s denoted by 
\[
\mathcal{C}=\{(x_n)_{n \in \N} \in \N^\N \mid \forall n \in \N, x_n \in \{0,1\} \}.
\]
The Cantor space is a compact metric space when equipped with the metric defined for all $x=(x_n)_{n \in \N}, y=(y_n)_{n \in \N} \in \mathcal{C}$ by 
\[
\dc(x,y)=\begin{cases}
0 & \text{: if } x=y\\
2^{-\min \{m \in \N \mid x_m \neq y_m\}} & \text{: otherwise}.
\end{cases}
\]
\end{definition}

The main subset that we will compare the Lipschitz seminorm $\ldc$ on $C(\mathcal{C})$ and the Lip-norm from \cite{Aguilar-Latremoliere15} is a certain dense subalgebra of $C(\mathcal{C})$. Thus, we now introduce notation for this subalgebra and list many facts from  \cite{Aguilar-Latremoliere15} that are important for our work, and are needed for defining the Lip-norm on $C(\mathcal{C})$ from \cite{Aguilar-Latremoliere15}.
\begin{notation}[{\cite[Section 7]{Aguilar-Latremoliere15}}]\label{n:dense-subalg}
Let $n \in \N$. We denote the $n^{th}$ coordinate evaluation map on $\mathcal{C}$ by 
\begin{equation*}
    \eta_n: (z_m)_{m \in \N} \in \mathcal{C} \longmapsto z_n \in \C
\end{equation*}
Also, define 
\begin{equation*}
    u_n = 2\eta_n - 1_{C(\mathcal{C})}
\end{equation*}
where $1_{C(\mathcal{C})}$ is the constant $1$ function on $\mathcal{C}$.

We note that $\eta_n, u_n \in C(\mathcal{C})$ $\eta_n^2=\eta_n$ and that the complex conjugate functions  $\overline{\eta_n}=\eta_n$ and $\overline{u_n}=u_n$, and $u_n^2=1_{C(\mathcal{C})}$. Furthermore, $\|\eta_n\|_\infty=\|u_n\|_\infty=1$, and $\eta_n(\mathcal{C}) =\{0,1\}$ and  $u_n(\mathcal{C}) =\{-1,1\}$.

Next, set $B_0=\emptyset$, and for each $n \in \N\setminus\{0\}$ set \[B_n=\left\{\prod_{j \in F} u_j \mid \emptyset \neq F\subseteq \{0, \ldots,n-1\}\right\}.\] Next, for each $n \in \N$, set $B_n'=\{1_{C(\mathcal{C})}\}\cup B_n$ and note that $|B_n'|=2^n$ since $C(\mathcal{C})$ is commutative and by the relations satisfied by the $u_n$'s  where $|B_n'|$ is the cardinality of $B_n'$ (see \cite[Notation 2.1.77]{Aguilar-Thesis} and \cite[Lemma 7.4]{Aguilar-Latremoliere15}).  Now, set 
\[A_n=\mathrm{span}\ B_n',\] which is a unital C*-subalgebra of $C(\mathcal{C})$ by  \cite[Lemma 7.4]{Aguilar-Latremoliere15}, where in \cite[Section 7]{Aguilar-Latremoliere15} the notation $\A_n$ was used   instead of $A_n$.  Also, note $\cup_{n \in \N} A_n$ is an  dense unital infinite-dimensional *-subalgebra of $C(\mathcal{C})$ and the set
\[
\left\{\prod_{j \in F} u_j \mid \emptyset \neq F \subset \N, F\text{ is finite} \right\}
\]
is a Hamel basis of $\cup_{n \in \N} A_n$ by \cite[Lemma 7.4]{Aguilar-Latremoliere15} and its proof.
\end{notation}
Now, that we have introduced the appropriate algebraic properties of $C(\mathcal{C})$ for our work, we introduce the analytical properties needed to build the Lip-norm from \cite{Aguilar-Latremoliere15}. In particular, we use a   state $\lambda \in S(C(\mathcal{C}))$ to build this Lip-norm. However, we note that $\lambda$ is built from the algebraic structure of $\mathcal{C}$ viewed as a compact group since $\lambda$ is induced by the Haar measure on this compact group (see \cite[Lemma 3.1.14]{Aguilar-Thesis} and \cite[Section 7]{Aguilar-Latremoliere15}).

\begin{lemma}[{\cite[Section 7]{Aguilar-Latremoliere15}}]\label{l:en}
The state $\lambda \in S(C(\mathcal{C}))$
of \cite[Notation 7.3]{Aguilar-Latremoliere15} satisfies \begin{enumerate}
            \item $\lambda\left(\prod_{j\in F}\eta_j\right)= 2^{-|F|}$ where $|F|$ represents the cardinality
            \item $\lambda\left(\prod_{j \in F}u_j\right)=0$
\end{enumerate}
for all $\emptyset \neq F \subset \N$ finite by \cite[Lemma 7.4]{Aguilar-Latremoliere15} and its proof. 

Furthermore, for each $n \in \N$, the continuous linear function
\[
E_n: C(\mathcal{C}) \rightarrow A_n
\]
of \cite[Theorem 3.5]{Aguilar-Latremoliere15} associated to $\lambda$ satisfies
\begin{enumerate}
    \item $E_n(f) = \sum_{a\in B_n^{'}} \lambda(fa)a$ for all $f \in C(\mathcal{C})$ by \cite[Expression (4.1) and Lemma 7.4]{Aguilar-Latremoliere15},
    \item $\|E_n(f)\|_{\infty} \leq \|f\|_{\infty}$  for all  $f \in C(\mathcal{C})$ by \cite[Definition 3.1]{Aguilar-Latremoliere15}
    \item $E_n(C(\mathcal{C}))=A_n$ by \cite[Theorem 3.5]{Aguilar-Latremoliere15}
    \item $E_n(a) = a,$ for all $a \in A_l$ such that $l \in \{0,...,n\}$
    \item $E_n(1_{C(\mathcal{C})}) = 1_{C(\mathcal{C})}$ by (4)
    \item $E_n(abc) = aE_n(b)c$   for all $a,c \in A_l, \text{where l} \in \{0,...,n\}$ by \cite[Definition 3.1]{Aguilar-Latremoliere15}
    \item If $k \in \{0,...,n\}$, then
        \begin{enumerate}
            \item If $n=0$, then    
                \begin{equation*}
                    E_0(u_k) = 0, \forall k \in \N
                \end{equation*} by proof of \cite[Theorem 7.5]{Aguilar-Latremoliere15}
            \item If $n \geq 1$, then 
                \begin{enumerate}
                    \item If $k \in \{0,...,n-1\}$
                        \begin{equation*}
                            E_n(u_k) = u_k
                        \end{equation*} by (4)
                    \item If $k \in \{n, n+1,...\}$
                        \begin{equation*}
                            E_n(u_k) = 0
                        \end{equation*}
                        by proof of \cite[Theorem 7.5]{Aguilar-Latremoliere15}.
                \end{enumerate}
            \end{enumerate}
\end{enumerate}
\end{lemma}
Finally, we define the Lip-norm $\lac$ on $C(\mathcal{C})$ from  \cite{Aguilar-Latremoliere15} that we will compare with $\ldc$.

\begin{theorem}[{\cite[Theorem 3.5 and Corollary 7.6]{Aguilar-Latremoliere15}}]\label{t:lac}
If   we define
\[
\lac(f)=\sup_{n \in \N}\left\{\frac{\|f-E_n(f)\|_\infty}{2^{-(n+1)}}\right\}
\]
for all $f \in C(\mathcal{C})$, then $(C(\mathcal{C}), \lac)$ is a $2$-Leibniz compact quantum metric space.  

Moreover,
\begin{enumerate}
    \item $\cup_{n \in \N} A_n \subseteq \dom{\lac}$,
    \item for all $n \in \N\setminus \{0\}$, it holds that 
    \[
    \lac(f)=\max_{k \in \{0, \ldots,n-1\}}\left\{\frac{\|f-E_k(f)\|_\infty}{2^{-(k+1)}} \right\}
    \]
    for all $f \in A_n$, 
    and 
    \item for all $x,y \in \mathcal{C},$ it holds that
    \[
    \Kantorovich{\lac}(\delta_x, \delta_y)= \dc(x,y)=\Kantorovich{\ldc}(\delta_x, \delta_y).
    \]
\end{enumerate}
\end{theorem}
It is the last expression that suggests that $\lac$ and $\ldc$ could agree on a dense subspace of $C(\mathcal{C}).$  Indeed, this expression serves as a main assumption of \cite[Theorem 8.1]{Rieffel99} that provides an equivalence for this agreement.  Thus, it is a main goal of this paper to show that $\lac$ and $\ldc$ disagree and we separate them on $\cup_{n \in \N} A_n$.

\section{Formulas for \texorpdfstring{$\ldc$ and $\lac$}{Lg}}\label{s:formulas}
Now, we will provide the main tools we use to separate $\lac$ and $\ldc$ on $\cup_{n \in \N}A_n $.  We do this by providing formulas for $\lac$ and $\ldc$ on each $A_n$.  Also, we show that $\lac$ and $\ldc$ agree on a Hamel basis for $\cup_{n \in \N} A_n$, which provides further evidence that $\lac$ and $\ldc$ could agree, but they do not as seen in Section \ref{s:separating}.
Our first task is to show that $\ldc$ and $\lac$ are even comparable. We already know that $\cup_{n \in \N} A_n \subseteq \dom{\lac}$ by Theorem \ref{t:lac}, so we will now show that $\cup_{n \in \N} A_n \subseteq \dom{\ldc}$.    

\begin{theorem}\label{t:domain}
It holds that $\cup_{n \in \N} A_n \subseteq \dom{\ldc}$. In particular, for all $n \in \N$, we have
\[
\ldc(u_n)=2^{n+1}< \infty.
\]
\end{theorem}
\begin{proof} Let $n \in \N.$
            Fix $x,y \in \mathcal{C}$ such that $x \neq y$. Then, 
            \begin{align*}
                |u_n(x)-u_n(y)|& = |2\eta_n(x)-1-2\eta_n(y)+1|\\
                &=2|\eta_n(x)-\eta_n(y)|.
            \end{align*}
        Since $x \neq y$, we know that there exists a smallest $k \in \N$ such that $x_k \neq y_k$. Therefore, 
            \begin{equation*}
              \dd_\mathcal{C}(x,y) = 2^{-k}
            \end{equation*}
            and $|x_k-y_k|=1$ since $x_k,y_k \in \{0,1\}$ and $x_k\neq y_k$.
        
        First, assume $k \in \{0,...,n\}$. If $k=n$,
            \begin{align*}
      \frac{|u_n(x)-u_n(y)|}{\dc(x,y)}& =  \frac{2|\eta_n(x)-\eta_n(y)|}{2^{-k}}=
                \frac{2|\eta_k(x)-\eta_k(y)|}{2^{-k}} \\
                &=\frac{2|x_k-y_k|}{2^{-k}} =\frac{2\cdot 1}{2^{-k}}=\frac{2}{2^{-n}} = 2\cdot2^n
            \end{align*}
        If $k>n$, then $x_n=y_n$ since $k$ is the first coordinate where $x$ and $y$ disagree. Hence,  $u_n(x)=u_n(y)$ by definition of $u_n$,  and thus 
        \begin{equation*}
             \frac{|u_n(x)-u_n(y)|}{\dc(x,y)}=\frac{|u_n(x)-u_n(y)|}{2^{-k}} =\frac{0}{2^{-k}}= 0. 
        \end{equation*} 
        Thus, if $n=0$, then we would be done. Next, assume $n>0$. The remaining case is: $0\leq k<n$.  Thus, $2^k < 2^n \implies \frac{1}{2^{-k}}< \frac{1}{2^{-n}}$ and therefore
        \begin{equation*}
            \frac{|u_n(x)-u_n(y)|}{\dc(x,y)}=\frac{2|\eta_n(x)-\eta_n(y)|}{2^{-k}} < \frac{2|\eta_n(x)-\eta_n(y)|}{2^{-n}}\leq \frac{2\cdot 1}{2^{-n}}
        \end{equation*}
        since $|\eta_n(x)-\eta_n(y)|=|x_n-y_n|\leq 1$ for any $n \in \N$ as $x_n,y_n \in \{0,1\}.$

        Hence, \[\frac{|u_n(x)-u_n(y)|}{d_c(x,y)}\leq \frac{2}{2^{-n}}\] for all $x,y \in \mathcal{C}, x \neq y$ and there exists $x,y \in \mathcal{C}$ such that $\frac{|u_n(x)-u_n(y)|}{d_c(x,y)}= \frac{2}{2^{-n}}$ since we may just choose $x,y \in \mathcal{C}$ such that the first coordinate they disagree at is $n$. Thus, 
        \begin{equation*}
            \ldc (u_n) =\sup_{x,y \in \mathcal{C}, x\neq y}\left\{ \frac{|u_n(x)-u_n(y)|}{\dc(x,y)}\right\} = 2\cdot2^n=2^{n+1} < \infty. 
        \end{equation*}
     Now, by the Leibniz rule and induction, we have that $\ldc(f)< \infty$, where $f$ is any finite product of $u_n$'s. Thus, since $\ldc$ is a seminorm and by induction, we have that $\ldc(f)<\infty$, where $f$ is any finite linear combination of finite products of $u_n$'s. Therefore $\ldc(f)< \infty$ for all $f \in \cup_{n \in \N} A_n,$ by definition of the $A_n$'s.
\end{proof}

Thus, we are now free to compare $\ldc$ and $\lac$ on $\cup_{n \in \N} A_n$. Now, by the proof of \cite[Theorem 7.5]{Aguilar-Latremoliere15}, we have that $\lac(u_n)=2^{n+1}=\ldc(u_n)$ for all $n \in \N$. It turns out that we can do much more and show that $\lac$ and $\ldc$ agree on all the elements of the Hamel basis of $\cup_{n\in \N} A_n$ given in Notation \ref{n:dense-subalg}.  The proof of Theorem \ref{t:basis-compare} follows a similar process as the proof of Theorem \ref{t:domain}, but requires some  different techniques that are crucial and acknowledge deeper structure. Thus, we first prove a lemma about the algebraic structure of the $E_n$'s which extends (7) of Lemma \ref{l:en} to finite products of $u_n$'s.

\begin{lemma}\label{l:en-vanish}
Let $k \in \N$. If $z \in \N\setminus \{0\}$ and $j_0, \ldots, j_z \in \N$ such that $j_0< \cdots <j_z,$ then 
\[
E_k(u_{j_0} \cdots u_{j_z})=\begin{cases}
u_{j_0} \cdots u_{j_z} & \text{: if } j_z <k\\
0_{C(\mathcal{C})} & \text{: if } j_z\geq k.  
\end{cases}
\]
\end{lemma}
\begin{proof}
First, if $j_z <k$, then $u_{j_0} \cdots u_{j_z}\in A_{k}$, and thus $E_k(u_{j_0} \cdots u_{j_z})=u_{j_0} \cdots u_{j_z}$ by (4) of Lemma \ref{l:en}.

Next, assume $j_z \geq k$.   Let $a\in B_k'$, then if $a=1_{C(\mathcal{C})}$, we have  $\lambda(u_{j_0} \cdots u_{j_z}a)=\lambda(u_{j_0} \cdots u_{j_z})=0$ by the first (2) of Lemma \ref{l:en}.  Next, by definition of $B_k'$, if $a=u_{n_0}\cdots u_{n_p}$, then $n_0, \ldots, n_p \in \N, n_0<\cdots< n_p <k\leq j_z$. Set $F=\{j_0, \ldots, j_z\} \triangle \{n_0, \ldots,n_p\},$ where $\triangle$ denotes symmetric difference.  By assumption, we have that $F \neq \emptyset$ since $j_z \in F$ as $j_z \neq u_{n_q}$ for all $ q \in \{0, \ldots, p\}$. Also, since $u_n^2=1_{C(\mathcal{C})}$ for all $n \in \N$ by Notation \ref{n:dense-subalg}, we have that \[u_{j_0} \cdots u_{j_z}a=\prod_{m \in F} u_m.\]  Therefore, as $F \neq \emptyset$, we have that   \[\lambda(u_{j_0} \cdots u_{j_z}a)=\lambda\left(\prod_{m \in F} u_m\right) =0\] by the first (2) of Lemma \ref{l:en}.  This exhausts all  elements in $B_k'$. Thus, we have
\[
E_k(u_{j_0} \cdots u_{j_z})=\sum_{a \in B_k'} \lambda(u_{j_0} \cdots u_{j_z}a)a=\sum_{a \in B_k'} 0\cdot a=0_{C(\mathcal{C})}
\]
by the second (1) of Lemma \ref{l:en}.
\end{proof}
\begin{theorem}\label{t:basis-compare}
For each $n \in \N$, it holds that
\[
\ldc(u_n)=2^{n+1}=\lac(u_n).
\]
And, for each $z \in \N\setminus \{0\}$, $j_0, \ldots, j_z \in \N$ such that $j_0<\cdots < j_z$, it holds that
\[
\ldc(u_{j_0}\cdots u_{j_z})=2^{j_z+1}=\lac(u_{j_0}\cdots u_{j_z}).
\]
\end{theorem}
\begin{proof}
The first equality is  provided by Theorem \ref{t:domain} and the proof of \cite[Theorem 7.5]{Aguilar-Latremoliere15}.

Next, let $z \in \N\setminus \{0\}$, $j_0, \ldots, j_z \in \N$ such that $j_0<\cdots < j_z$. 

We will first consider $\ldc(u_{j_0}\cdots u_{j_z}).$ Let $x,y \in \mathcal{C}$ such that $x \neq y.$ Thus, there smallest $k \in \N$ such that $x_k\neq y_k$, and thus $\dc(x,y)=2^{-k}.$

\begin{case}
Assume that $j_z <k$.
\end{case} Then $j_0, \ldots, j_z<k$.  Hence, $x_{j_0}=y_{j_0}, \ldots, x_{j_z}=y_{j_z}$ since $k$ is the first coordinate where $x$ and $y$ disagree. Thus,   $u_{j_0}(x)=u_{j_0}(y), \ldots, u_{j_z}(x)=u_{j_z}(y) $ by definition of  $u_{j_0}, \ldots, u_{j_z}.$ Therefore, 
\[\frac{|u_{j_0}\cdots u_{j_z}(x)-u_{j_0}\cdots u_{j_z}(y)|}{\dc(x,y)}=\frac{|u_{j_0}(x)\cdots u_{j_z}(x)-u_{j_0}(y)\cdots u_{j_z}(y)|}{\dc(x,y)}=0.\]

\begin{case}
Assume that $k=j_z$.
\end{case}    Hence, $x_{j_0}=y_{j_0}, \ldots, x_{j_{z-1}}=y_{j_{z-1}},  x_{j_z}\neq y_{j_z}$ since $k=j_z$ is the first coordinate where $x$ and $y$ disagree. Thus,   $u_{j_0}(x)=u_{j_0}(y), \ldots, u_{j_{z-1}}(x)=u_{j_{z-1}}(y) $ and $u_{j_z}(x)\neq u_{j_z}(y)$, and thus $u_{j_z}(x)=-u_{j_z}(y)$ by definition of   $u_{j_0}, \ldots, u_{j_z}$ and the fact that the range of these functions is $\{-1,1\}.$  Therefore, we have
\begin{align*}
   & \frac{|u_{j_0}\cdots u_{j_z}(x)-u_{j_0}\cdots u_{j_z}(y)|}{\dc(x,y)}\\
   & = \frac{|u_{j_0}(x)\cdots u_{j_z-1}(x) u_{j_z}(x)-u_{j_0}(y)\cdots u_{j_z-1}(y) u_{j_z}(y)|}{2^{-j_z}}\\
    & = \frac{|u_{j_0}(x)\cdots u_{j_z-1}(x) u_{j_z}(x)-u_{j_0}(x)\cdots u_{j_z-1}(x)(- u_{j_z}(x))|}{2^{-j_z}}\\
    & =  \frac{|u_{j_0}(x)\cdots u_{j_z-1}(x) u_{j_z}(x)+u_{j_0}(x)\cdots u_{j_z-1}(x)  u_{j_z}(x)|}{2^{-j_z}}\\
    &= \frac{2|u_{j_0}(x)\cdots u_{j_z-1}(x) u_{j_z}(x)|}{2^{-j_z}}  = \frac{2\cdot 1 }{2^{-j_z}}=2^{j_z+1}
\end{align*}
since $u_{j_0}(x)\cdots u_{j_z-1}(x) u_{j_z}(x) \in \{-1,1\}$. Such  $x,y\in \mathcal{C}$ exist and thus $2^{j_z+1}$ does exist in $\left\{ \frac{|u_{j_0}\cdots u_{j_z}(x)-u_{j_0}\cdots u_{j_z}(y)|}{\dc(x,y)} \mid x,y \in \mathcal{C}, x \neq y\right\}.$

\begin{case}
Assume that $k <j_z.$
\end{case}  Then, we have $2^k<2^{j_z},$ which implies $\frac{1}{2^{-k}}< \frac{1}{2^{-j_z}}.$ So,
\begin{align*}
   \frac{|u_{j_0}\cdots u_{j_z}(x)-u_{j_0}\cdots u_{j_z}(y)|}{\dc(x,y)}
   & = \frac{|u_{j_0}\cdots u_{j_z}(x)-u_{j_0}\cdots u_{j_z}(y)|}{2^{-k}}\\
   & < \frac{|u_{j_0}\cdots u_{j_z}(x)-u_{j_0}\cdots u_{j_z}(y)|}{2^{-j_z}}\\
   & \leq \frac{|u_{j_0}\cdots u_{j_z}(x)|+|u_{j_0}\cdots u_{j_z}(y)|}{2^{-j_z}}\\
   & = \frac{1+1}{2^{-j_z}}=2^{j_z+1}.
\end{align*}
Therefore, 
\[
\ldc(u_{j_0}\cdots u_{j_z})=\sup_{x,y \in \mathcal{C}, x\neq y}\left\{ \frac{|u_{j_0}\cdots u_{j_z}(x)-u_{j_0}\cdots u_{j_z}(y)|}{\dc(x,y)}\right\}=2^{j_z+1}.
\]

Next, we consider $\lac(u_{j_0}\cdots u_{j_z}).$ Now, $u_{j_0}\cdots u_{j_z} \in A_{j_z+1}$, and thus by Theorem \ref{t:lac}, we have that
\[
\lac(u_{j_0}\cdots u_{j_z})=\max_{k \in \{0, \ldots, j_z\}} \left\{\frac{\|u_{j_0}\cdots u_{j_z}-E_k( u_{j_0}\cdots u_{j_z})\|_\infty}{2^{-(k+1)}} \right\}.
\]
By Lemma \ref{l:en-vanish}, we have   $E_{j_z}( u_{j_0}\cdots u_{j_z})=0_{C(\mathcal{C})}$ and $E_k( u_{j_0}\cdots u_{j_z})= u_{j_0}\cdots u_{j_z}$ for all $k \in \{0, \ldots, j_z-1\}.$  Hence
\[
\lac(u_{j_0}\cdots u_{j_z})= \frac{\|u_{j_0}\cdots u_{j_z}-0_{C(\mathcal{C})}\|_\infty}{2^{-(j_z+1)}} =\frac{1}{2^{-(j_z+1)}}=2^{j_z+1}
\]
since $u_{j_0}\cdots u_{j_z}(\mathcal{C})=\{-1,1\}$ and the definition of the norm $\|\cdot\|_\infty$.
\end{proof}
From this, we can see that $\ldc$ and $\lac$ agree on $A_0$ and $A_1.$ 
\begin{corollary}\label{c:l-agree}
If $f \in A_1$, then $f=\alpha_01_{C(\mathcal{C})}+\alpha_1 u_0$ for some $\alpha_0, \alpha_1 \in \C$ and 
\[
\lac(f)=2|\alpha_1|=\ldc(f),
\]
and thus $\lac=\ldc$ on $A_1$ and $A_0$.
\end{corollary}
\begin{proof}
Let $f \in A_1$, then by definition of $A_1$ and $B_1'$ there exist $\alpha_0, \alpha_1\in \C$ such that $f=\alpha_01_{C(\mathcal{C})}+\alpha_1 u_0$.

Now, since $\lac(\alpha_01_{C(\mathcal{C})})=0$ by Definition \ref{d:cqms}, we have by Theorem \ref{t:basis-compare} and since $\lac$ is a seminorm
\begin{align*}
    2|\alpha_1|&=|\alpha_1|\cdot 2^{0+1}=|\alpha_1|\lac(u_0)=\lac(\alpha_1u_0)\\
    &=  \lac(\alpha_01_{C(\mathcal{C})}+\alpha_1u_0-\alpha_01_{C(\mathcal{C})})\\
    &\leq \lac(\alpha_01_{C(\mathcal{C})}+\alpha_1u_0)+\lac(\alpha_01_{C(\mathcal{C})})\\
    &= \lac(f)\\
    &= \lac(\alpha_01_{C(\mathcal{C})}+\alpha_1u_0) \leq \lac(\alpha_01_{C(\mathcal{C})})+\lac(\alpha_1u_0)=\lac(\alpha_1u_0)\\
    &= 2|\alpha_1|.
\end{align*}
Thus $\lac(f)= 2|\alpha_1|.$ The same argument shows this is true for $\ldc(f)$ by Definition \ref{d:cqms} and Theorem  \ref{t:basis-compare} and the fact that $\ldc$ is a seminorm.  Thus they agree on $A_1$ and on $A_0$ since $A_0\subseteq A_1.$
\end{proof}

Thus far, we have been able to find formulas for the elements of the Hamel basis of Notation \ref{n:dense-subalg}, but now, we will develop formulas for $\lac$ and $\ldc$ on $A_n$ for all $n \geq 2$ that are built using the basis elements. We note that Corollary \ref{c:l-agree} already provided a formula for $\lac$ and $\ldc$ on $A_0$ and $A_1$. This will use some of the machinery developed in  the proof of Theorem \ref{t:basis-compare} along with the following technical lemma that will help us better understand the behavior of the difference quotients in the definition of $\ldc$.
\begin{lemma}\label{l:diff-basis}
Let   $z \in \N\setminus \{0\}$ and $j_0, \ldots, j_z \in \N$ such that $j_0<\cdots <j_z$. Let $k \in \N$ such that $k<j_z$. Assume $x,y \in \mathcal{C}$ such that $\dc(x,y)=2^{-k}$.
\begin{enumerate}
\item  $u_{j_z}(x)-u_{j_z}(y)=0$ if and only if  $u_ku_{j_z}(x)-u_ku_{j_z}(y)\in \{-2,2\},$ and

 $u_{j_z}(x)-u_{j_z}(y)\in \{-2,2\}$ if and only if   $u_ku_{j_z}(x)-u_ku_{j_z}(y)=0.$
    \item If $k=j_m$ for some $m \in \{0, \ldots, z-1\}$, then   
    \begin{enumerate}
    \item $\left(\prod_{l=0}^z u_{j_l} \right)(x)-\left(\prod_{l=0}^z u_{j_l} \right)(y)=0 $ if and only if
    \[
    \left(\prod_{l=0, l\neq m}^z u_{j_l} \right)(x)-\left(\prod_{l=0, l\neq m}^z u_{j_l} \right)(y)\in \{-2,2\}, 
    \] and
    \item $\left(\prod_{l=0}^z u_{j_l} \right)(x)-\left(\prod_{l=0}^z u_{j_l} \right)(y)\in \{-2,2\} $ if and only if 
    \[
    \left(\prod_{l=0, l\neq m}^z u_{j_l} \right)(x)-\left(\prod_{l=0, l\neq m}^z u_{j_l} \right)(y)=0. 
    \]
    \end{enumerate}
    \item If $j_m<k<j_{m+1}\leq j_z$ for some $m \in \{0, \ldots, z-1\}$, then  
    \begin{enumerate}
    \item $\left(\prod_{l=0}^z u_{j_l} \right)(x)-\left(\prod_{l=0}^z u_{j_l} \right)(y)=0 $ if and only if
    \[
    u_k\left(\prod_{l=0}^z u_{j_l} \right)(x)-u_k\left(\prod_{l=0}^z u_{j_l} \right)(y)\in \{-2,2\}, 
    \] and
    \item $\left(\prod_{l=0}^z u_{j_l} \right)(x)-\left(\prod_{l=0}^z u_{j_l} \right)(y)\in \{-2,2\} $ if and only if
    \[
    u_k\left(\prod_{l=0}^z u_{j_l} \right)(x)-u_k\left(\prod_{l=0}^z u_{j_l} \right)(y)=0. 
    \]
    \end{enumerate}
\end{enumerate}
Moreover, if we set
\begin{align*}
    C_k&=  \{a \in\cup_{n \in \N} B_n \mid a=u_p, p \in \N, p >k \\
    & \quad \quad \quad \quad \quad \quad \quad \quad \quad \text{ or } a=u_{j_0}\cdots u_{j_z}, j_0< \cdots < j_z \in \N, j_z>k\},
\end{align*}
 and set 
 \[
 C_k'=\{a \in C_k \mid a(x)-a(y) \in \{-2,2\}\}
 \]
 and $C_k''=C_k\setminus C_k'=\{a \in C_k \mid a(x)-a(y)=0\},$
 then the map
 \[
 \Delta_k:a \in C_k'\mapsto u_ka \in C_k''
 \]
 is a well-defined bijection.
\end{lemma}
\begin{proof}
(1) By definition of the $u_n's$ in Notation \ref{n:dense-subalg}, we have  that  $u_{j_z}(x)-u_{j_z}(y) \in \{-2,0,2\}$ and $ u_ku_{j_z}(x)-u_ku_{j_z}(y)\in \{-2,0,2\}$.  Now assume that $u_{j_z}(x)-u_{j_z}(y)=0$, then $u_{j_z}(x)+u_{j_z}(y)\neq 0,$ which implies that $u_{j_z}(x)+u_{j_z}(y) \in \{-2,2\}$ since $u_{j_z}(x), u_{j_z}(y) \in \{-1,1\}.$  Now, $u_k(x)=-u_k(y)$ by Case 2 of Theorem \ref{t:basis-compare}.  Hence
\begin{align*}
u_ku_{j_z}(x)-u_ku_{j_z}(y)&=u_k(x)u_{j_z}(x)+u_k(x)u_{j_z}(y)\\
& = u_k(x)(u_{j_z}(x)+u_{j_z}(y)) \in \{-2,2\}.
\end{align*}
The other direction is similar. And, the other if and only if is simply the negation of the first since the values considered are only $\{-2,0,2\}.$

(2) If $k=j_m$, then by the same argument as Case 2 of Theorem \ref{t:basis-compare}, we have that $u_{j_0}(x)=u_{j_0}(y), \ldots, u_{j_{m-1}}(x)=u_{j_{m-1}}(y), u_{j_m}(x)=-u_{j_m}(y)$. Hence
\begin{align*}
    r_1&=u_{j_0}\cdots u_{j_z}(x)-u_{j_0}\cdots u_{j_z}(y)\\
    & = u_{j_0}\cdots u_{j_{m-1}}(x)(u_{j_m}(x)\cdots u_{j_z}(x)-u_{j_m}(y)\cdots u_{j_z}(y))\\
    & = u_{j_0}\cdots u_{j_{m-1}}(x)(u_{j_m}(x)\cdots u_{j_z}(x)+u_{j_m}(x)\cdots u_{j_z}(y))\\
    & = u_{j_0}\cdots u_{j_m}(x)(u_{j_{m+1}}(x)\cdots u_{j_z}(x)+u_{j_{m+1}}(y)\cdots u_{j_z}(y))
\end{align*}
and
\begin{align*}
    r_2&=u_{j_0}\cdots u_{j_{m-1}}u_{j_{m+1}}\cdots u_{j_z}(x)-u_{j_0}\cdots u_{j_{m-1}}u_{j_{m+1}}\cdots u_{j_z}(y)\\
    & = u_{j_0}\cdots u_{j_{m-1}}(x)(u_{j_{m+1}}(x)\cdots u_{j_z}(x)-u_{j_{m+1}}(y)\cdots u_{j_z}(y))
\end{align*}
Now, we note that $u_{j_0}\cdots  u_{j_m}(x),u_{j_0}\cdots u_{j_{m-1}}(x)\in \{-1,1\}$, and thus not zero. Again by definition of the $u_n$'s, we have    $u_{j_{m+1}}(x)\cdots u_{j_z}(x)+u_{j_{m+1}}(y)\cdots u_{j_z}(y) \in \{-2,0,2\}$ and $ u_{j_{m+1}}(x)\cdots u_{j_z}(x)-u_{j_{m+1}}(y)\cdots u_{j_z}(y) \in \{-2,0,2\} $.

Thus, if $r_1=0,$ then we have $u_{j_{m+1}}(x)\cdots u_{j_z}(x)+u_{j_{m+1}}(y)\cdots u_{j_z}(y)=0,$ and  thus $u_{j_{m+1}}(x)\cdots u_{j_z}(x)-u_{j_{m+1}}(y)\cdots u_{j_z}(y)\neq 0$ since $u_{j_{m+1}}(x)\cdots u_{j_z}(x)\neq 0$ and $u_{j_{m+1}}(y)\cdots u_{j_z}(y) \neq0$, which implies that $r_2 \neq 0$ since $ u_{j_0}\cdots u_{j_{m-1}}(x)\neq 0$, and thus $r_2 \in \{-2,2\}.$

Next, assume we have $r_2 \in \{-2,2\}$. Set $v=u_{j_{m+1}}(x)\cdots u_{j_z}(x) \in \{-1,1\}$ and $w=u_{j_{m+1}}(y)\cdots u_{j_z}(y) \in \{-1,1\}$. First, consider $r_2=-2$. Then,   $v-w \in \{-2,2\}$ since $ u_{j_0}\cdots u_{j_{m-1}}(x)\neq 0$. If $v-w=-2,$ then $v+w=-2+w+w=2(w-1).$ If $w=-1$, then $v+w=-4,$ which is a contradiction since $v,w \in \{-1,1\}.$ Hence $w=1,$ and thus $v+w=0,$ which implies $r_1 =0.$ Now, if $v-w=2$, then $v+w=2+w+w=2(w+1)$. If $w=1,$ then $v+w=4,$ which is a contradiction since $v,w \in \{-1,1\}.$ Hence $w=-1$, and thus $v+w=0,$ which implies that $r_1=0.$  The case when $r_2=2$ is the same proof and provides $r_1=0$ as well. Hence, if $r_2\in \{-2,2\}$, then $r_1=0.$  This concludes (2)(a). 

(2)(b) This is simply the negation of (2)(a) since the values considered are only $\{-2,0,2\}.$

(3) This is the same argument as (2).

Now, we establish the bijection at the end of the theorem. Note that $C_k''=\{a \in C_k\mid a(x)-a(y)=0\}$ since $a(x)-a(y) \in \{-2,0,2\}$ by previous arguments. First, we show well-defined. Let $a \in C_k'$. Now, $a(x)-a(y)$ must be of the form given in (1), (2), (3). If $u_k$ is not part of the product forming $a$, then $a(x)-a(y)$ falls under the second line of (1) or (3)(b), and in either case, we have $u_ka \in C_k''.$  Now, if $u_k$ is part of the product forming $a$, then $a(x)-a(y)$ falls under (2)(b). Hence $a=u_{j_0} \cdots u_{k}\cdots u_{j_z}=u_{j_0} \cdots u_{j_m}\cdots u_{j_z}$.  Thus, by Notation \ref{n:dense-subalg}, we have 
\begin{equation}\label{eq:uk-vanish}
u_ka=u_{j_0} \cdots u_{k }^2\cdots u_{j_z}=u_{j_0} \cdots 1_{C(\mathcal{C})}\cdots u_{j_z}=u_{j_0}\cdots u_{j_{m-1}}u_{j_{m+1}}
\cdots u_{j_z}.\end{equation}
Thus, by (2)(b), we have that $u_ka(x)-u_ka(y)=0,$ which implies that $u_ka \in C_k''.$ Hence, the map $\Delta_k$ is well-defined. 

Now, let's establish surjectivity. Let $b \in C_k''.$ Thus $b(x)-b(y)=0$. If $u_k$ is not part of the product forming $b$, then $b(x)-b(y)$ falls under the first line of (1) or (3)(a). Hence, $u_kb \in C_k'$ in either case, and $\Delta_k(u_kb)=u_k(u_kb)=u_k^2b=b.$  Next, if $u_k$ is part of the product forming $b$, then $b(x)-b(y)$ falls under (2)(a), and the same argument of Expression \eqref{eq:uk-vanish} shows that $u_kb\in C_k'$, and thus, $\Delta_k(u_kb)=b$ as in the previous case.  Thus $\Delta_k$ is a surjection.

Next, injectivity.  Let $a,a' \in C_k'$ and $\Delta_k(a)=\Delta_k(a')$.  Then $u_ka=u_ka'$. Thus $u_k(u_ka)=u_k(u_ka)$ implies $u_k^2a=u_k^2a'$ implies $a=a'.$
\end{proof}

We note that we only consider formulas for elements in $A_n$ that are linear combinations of elements in $B_n$ rather than $B_n'$ since both $\lac$ and $\ldc$ are not affected by $1_{C(\mathcal{C})}$ by the argument in Corollary \ref{c:l-agree}. 
\begin{theorem}\label{t:ldc-formula}
Let $n \in \N, n \geq 2$. Let $f \in A_n$  such that  $f=\sum_{a \in B_n}\alpha_a a$, where  $\alpha_a \in \C$ for all $a \in B_n$.

Next, define 
\[\mathcal{C}_n=\{x \in \mathcal{C} \mid \forall k \geq n, x_k=0 \},\] which has $2^n$ elements. Let $x,y \in \mathcal{C}_n$ and denote $k_{x,y}=-\log_2 \dc(x,y)$.  Define \[\sigma_{x,y}=\{ a \in B_n \mid a(x)-a(y) \neq 0\}.\]

We then have
\[
\ldc(f)=\max_{x,y \in \mathcal{C}_n, x \neq y} \left\{ 2^{k_{x,y}+1}\left|\sum_{a \in \sigma_{x,y}} \pm_{a,x,y} \alpha_a\right| \right\},
\]
where $\pm_{a,x,y}$ is the sign  of $a(x)-a(y)$  and we note that $\{k_{x,y} \mid x,y \in \mathcal{C}_n\}=\{0, \ldots, n-1\}$ and the cardinality of $\sigma_{x,y}$ is $|\sigma_{x,y}|=2^{n-1}$.
\end{theorem}
\begin{proof}
By definition of $A_n,$ we have that $f(x)=f(y)$ for all $x,y \in \mathcal{C}$ such that $\dc(x,y)\geq n.$ Hence, we need only consider $x,y \in \mathcal{C}_n$ with $x \neq y$. Also, the cardinality $|\mathcal{C}_n|=|\mathrm{Powerset}(\{0, \ldots,n-1\})|=2^n.$ For ease of notation in the rest of the proof, set $k=k_{x,y}$, and we note that $k \in \{0, \ldots, n-1\}$ by definition of  $\mathcal{C}_n.$  If $k=0$, define $Z_0=\emptyset$ and if $k>0$, define
\begin{align*}
Z_k&= \ \ \{a \in B_n \mid a=u_p, p \in \{0, \ldots, k-1\} \\
&\quad \quad  \quad \quad \quad \quad \quad \text{ or } a=u_{j_0}\cdots u_{j_z}, j_0< \cdots< j_z \in \N, j_z<k\}.
\end{align*}
 Note by Case 1 of Theorem \ref{t:basis-compare}, we have that $a(x)-a(y)=0$ for all $a \in Z_k.$ Since $C(\mathcal{C})$ is commutative, we have that the cardinality
 \[
 |Z_k|=|\mathrm{Powerset}(\{0, \ldots, k-1\})|-1=2^k-1.
 \]
 Next, define
 \[
 S_k=\{a \in B_n\mid a=u_p, p=k \text{ or } a=u_{j_0}\cdots u_{j_z}, j_0<\cdots <j_z \in \N, j_z=k\}
 \]
 We note $Z_k \cap S_k=\emptyset.$ Also, 
 \[
 S_k=\{u_k\}\cup \{a \in B_n \mid a=u_{j_0}\cdots u_{j_z}, j_0<\cdots <j_z \in \N, j_z=k  \},
 \]
 where the two sets are disjoint.  Thus, similarly to $Z_k$, the cardinality
 \[
 |S_k|=|\{u_k\}|+\mathrm{Powerset}(\{0, \ldots, k-1\})|-1=1+2^k-1=2^k
 \]
 By Case 2 of Theorem \ref{t:basis-compare}, we have that $a(x)-a(y)\neq 0$ and thus $a(x)-a(y)\in \{-2,2\}$ since $a(x),a(y) \in \{-1,1\}$ by definition of the $u_n$'s.  
 
 Next, set
 \[
 C_k=B_n\setminus (Z_k\cup S_k).
 \]
 By disjoint sets, we have the cardinality, 
 \[
 |C_k|=|B_n|-(|Z_k|+|S_k|)=2^n-1-(2^k-1++2^k)=2^n-2^{k+1}.
 \]
 Also, note 
 \begin{align*}
 C_k&=\{a \in B_n \mid a=u_p, p \in \N, p >k\\
 &  \quad  \quad \quad \quad \quad \quad \text{ or } a=u_{j_0}\cdots u_{j_z}, j_0< \cdots < j_z \in \N, j_z>k\}.
 \end{align*}
 Now, define
 \[
 C_k'=\{a \in C_k \mid a(x)-a(y) \in \{-2,2\}\}.
 \]
 and $C_k''=C_k \setminus C_k'=\{a \in C_k \mid a(x)-a(y) =0\}$ of Lemma \ref{l:diff-basis}. Also, by Lemma \ref{l:diff-basis}, we have a bijection between $C_k'$ and $C_k''$. Hence, the cardinality
 \[
 |C_k'|=|C_k''|.
 \]
 Therefore, since all sets considered are finite as $B_n$ is finite, we have $|C_k'|=|C_k|-|C_k''|=|C_k|-|C_k'|,$ which implies $2|C_k'|=|C_k|$ which implies
 \[
 |C_k'|=\frac{1}{2} |C_k|=\frac{1}{2}(2^n-2^{k+1})=2^{n-1}-2^k.
 \]
 Now, we have that 
 \[
 S_k\cup C_k'=\{a \in B_n \mid a(x)-a(y)\neq 0\}=\{a \in B_n \mid a(x)-a(y)\in \{-2,2\}\}
 \]
 and are disjoint by construction. So, we set
 \[
 \sigma_{x,y}=S_k\cup C_k'
 \]
 and thus have, the cardinality
 \[
 |\sigma_{x,y}|=|S_k|+|C_k'|=2^k+2^{n-1}-2^k=2^{n-1}.
 \]
 Now
\begin{align*}
f(x)-f(y)&=\sum_{a \in B_n}\alpha_a(a(x)-a(y))=\sum_{a \in \sigma_{x,y}}\alpha_a(a(x)-a(y))\\
&=\sum_{a \in \sigma_{x,y}}\alpha_a(a(x)-a(y))=\sum_{a \in \sigma_{x,y}}\alpha_a(\pm_{a,x,y}2)\\
& =2\sum_{a \in \sigma_{x,y}}\pm_{a,x,y}\alpha_a,
\end{align*}
where $\pm_{a,x,y}$ is the sign of  $a(x)-a(y)\in \{-2,2\}$. 
Hence,
\[
\frac{|f(x)-f(y)|}{\dc(x,y)}=\frac{2\left|\sum_{a \in \sigma_{x,y}}\pm_{a,x,y}\alpha_a\right|}{2^{-k_{x,y}}}=2^{k_{x,y}+1}\left|\sum_{a \in \sigma_{x,y}}\pm_{a,x,y}\alpha_a\right|.
\]
Therefore, as $\mathcal{C}_n$ is finite, the proof is complete. 
\end{proof}

Next, we find a similar formula for $\lac$, which will reveal some important and crucial differences between the behavior of $\lac$ and $\ldc$.

\begin{theorem}\label{t:lac-formula}
Let $n \in \N, n \geq 2$. Let $f \in A_n$  such that  $f=\sum_{a \in B_n}\alpha_a a$, where  $\alpha_a \in \C$ for all $a \in B_n$.

Next, define 
\[\mathcal{C}_n=\{x \in \mathcal{C} \mid \forall k \geq n, x_k=0 \},\] which has $2^n$ elements. For each $k \in \{0, \ldots, n-1\}$, set
\[
\rho_k=\left\{a \in B_n \mid a-E_k(a)\neq 0_{\mathcal{C}}\right\}=\{a \in B_n \mid E_k(a)=0_{C(\mathcal{C})}\}.
\]

We then have
\[
\lac(f)=\max_{k \in \{0,\ldots, n-1\}} \left\{ \max_{x \in \mathcal{C}_n} \left\{2^{k +1}\left|\sum_{a \in \rho_k} \pm_{a,x} \alpha_a\right| \right\}\right\},
\]
where $\pm_{a,x}$ is the sign of $a(x)$   and we note that   the cardinality of $\rho_k$ is $|\rho_k|=2^n-2^k$.
\end{theorem}
\begin{proof}
The cardinality of $\mathcal{C}_n$ was already determined in Theorem \ref{t:ldc-formula}. Let $k \in \{0, \ldots, n-1\}.$ Since $f \in A_n$ and $E_k(f)\in A_k\subseteq A_n$ for all $k \in \{0, \ldots, n-1\}$, we have that $(f-E_k(f))(x)$ for $x \in \mathcal{C}$ is only determined by the values $x_0, \ldots, x_{n-1}$ by definition of the $u_n$'s.  Hence, 
\[
\|f-E_k(f)\|_\infty=\sup_{x \in \mathcal{C}} |(f-E_k(f))(x)|= \max_{x \in \mathcal{C}_n} |(f-E_k(f))(x)|.
\]
Set \[Z_k=\{a \in B_n\mid a-E_k(a)=0_{C(\mathcal{C})}\}.\]

Note that  $1_{C(\mathcal{C})} \not\in B_n$ and thus if $a \in B_n,$ then $a=u_{j_0}\cdots u_{j_z}$ for $j_0<\cdots < j_z \leq n-1\in \N$ or $a=u_p$ for some $p \in \{0, \ldots, n-1\}$. Thus $E_0(u_p)=0_{C(\mathcal{C})}$ by (7)(a) of Lemma \ref{l:en} and $E_0(u_{j_0}\cdots u_{j_z})=0_{C(\mathcal{C})}$  by Lemma \ref{l:en-vanish}. In either case we have that $a-E_0(a)=a\neq 0_{C(\mathcal{C})}\}.$ Therefore $Z_0=\emptyset.$  

Next, assume that $k>0$. By a similar argument, we have by Lemma \ref{l:en-vanish} and (7)(b)(i) of Lemma \ref{l:en} that
\begin{align*}
Z_k&=
 \{a \in B_n \mid a=u_p, p \in \{0, \ldots, k-1\}\\
 & \quad \quad \quad \quad \quad \quad  \text{ or } a=u_{j_0}\cdots u_{j_z}, j_0< \cdots< j_z \in \N, j_z<k\}.
\end{align*}
By the proof of Theorem \ref{t:ldc-formula}, we have that the cardinality
\[
|Z_k|=2^k-1.
\]
Now, set
\[
\rho_k=B_n\setminus Z_k=\{a \in B_n \mid a-E_k(a)\neq 0_{C(\mathcal{C})}\}=\{a \in B_n \mid E_k(a)=0_{C(\mathcal{C})}\}
\]
since $E_k(a) \in \{a, 0_{C(\mathcal{C})}\}$ by Lemma \ref{l:en-vanish} and (7) of Lemma \ref{l:en}. 

Next, the cardinality
\[
|\rho_k|= |B_n|-|Z_k|=2^n-1-(2^k-1)=2^n-2^k.
\]
Next, if $x \in \mathcal{C}_n$, we have by linearity of $E_k$ that
\begin{align*}
    f-E_k(f)= \sum_{a \in B_n}\alpha_a (a-E_k(a))=\sum_{a \in \rho_k}\alpha_a (a-E_k(a))= \sum_{a \in \rho_k}\alpha_a a.
\end{align*}
Thus, by the beginning of the proof we have
\begin{align*}
\|f-E_k(f)\|_\infty& = \max_{x \in \mathcal{C}_n} |(f-E_k(f))(x)|= \max_{x \in \mathcal{C}_n} \left|\left(\sum_{a \in \rho_k}\alpha_a a\right)(x)\right|\\
&=\max_{x \in \mathcal{C}_n}\left| \sum_{a \in \rho_k}\alpha_a a(x) \right|= \max_{x \in \mathcal{C}_n}\left| \sum_{a \in \rho_k}\pm_{a,x} \alpha_a  \right|
\end{align*}
since $a(x) \in \{-1,1\}$ by Notation \ref{n:dense-subalg}, and $\pm_{a,x}$ is the sign  of  $a(x)\in \{-1,1\}$. Hence by Theorem \ref{t:lac}, we have
\begin{align*}
    \lac(f)&=\max_{k \in \{0, \ldots, n-1\}}\left\{\frac{\|f-E_k(f)\|_\infty}{2^{-(k+1)}}\right\}\\
    &= \max_{k \in \{0, \ldots, n-1\}} \left\{ \max_{x \in \mathcal{C}_n}\left\{2^{k+1}\left| \sum_{a \in \rho_k}\pm_{a,x} \alpha_a  \right|\right\}\right\},
\end{align*}
which completes the proof.
\end{proof}

\section{Separating   \texorpdfstring{$\ldc$ and $\lac$}{Lg}}\label{s:separating}
  Theorems \ref{t:ldc-formula} and Theorem \ref{t:lac-formula} provide us with a general idea of the structure of these Lip-norms with respect to the structure of the dense subalgebra $\cup_{n \in \N} A_n.$ However, these formulas also gift insight into the differences between $\ldc$ and $\lac$.  In particular, the cardinality between $\sigma_{x,y}$ and $\rho_k$. The cardinality of  $\sigma_{x,y}$ on depends on the dimension of the algebra $A_n$ even though the set  $\sigma_{x,y}$ is built from the first coordinate $x$ and $y$ disagree.  However, the cardinality of  $\rho_k$ depends on the dimension of $A_n$ and the dimension of the  space $E_k$ projects onto, $A_k$. Therefore, $\lac$ captures more information from the coefficients of the element $f$ being entered into the Lip-norm than $\ldc.$  We will see that this happens at $A_2$ in comparing Theorem \ref{t:ldc2-formula} and Theorem \ref{t:lac2-formula}, where only pairs of coefficients are consider in the formula for $\ldc$ whereas pairs and triples of coefficients are consider in $\lac.$ So, the hope is that we can separate $\ldc$ and $\lac$ already on $A_2$ without having to go to higher dimension.  In Theorem \ref{t:lip-sep}, we accomplish this and provide many elements that separate $\lac$ and $\ldc$ on $A_2,$ but first, we take a closer look at our formulas on $A_2.$ As seen in the proof of Corollary \ref{c:l-agree}, we do not need to consider $1_{C(\mathcal{C})}$ in our calculations since the seminorms $\ldc$ and 
$\lac$ vanish on $1_{C(\mathcal{C})}$.
\begin{theorem}\label{t:ldc2-formula}
Let $f \in A_2$ such that $f=\alpha_0 u_0+\alpha_1 u_1+\alpha_2 u_0 u_1$ for some $\alpha_0, \alpha_1, \alpha_2 \in \C$.

It holds that
\[
\ldc(f)=\max\left\{\begin{array}{l}
2|\alpha_0-\alpha_2|,\ 2|\alpha_0-\alpha_1|, \  2|\alpha_0+\alpha_1|,\ 2|\alpha_0+\alpha_2|, \\
 4|\alpha_1-\alpha_2|, \ 4|\alpha_1+\alpha_2|
\end{array} \right\}.
\]
\end{theorem}
\begin{proof}
Note $B_2=\{u_0, u_1, u_0u_1\}$ by Notation \ref{n:dense-subalg}. By Theorem \ref{t:ldc-formula}, we have that 
\[
\ldc(f)=\max_{x,y \in \mathcal{C}_2, x \neq y} \left\{ 2^{k_{x,y}+1}\left|\sum_{a \in \sigma_{x,y}} \pm_{a,x,y} \alpha_a\right| \right\},
\]
where we set $\alpha_{u_0}=\alpha_0, \alpha_{u_1}=\alpha_1, \alpha_{u_0u_1}=\alpha_2,  $ and for all $x,y \in \mathcal{C}_2, x \neq y$, we have  $\sigma_{x,y}=\{a \in B_2\mid a(x)-a(y) \neq 0\}$ and $\pm_{a,x,y}$ is the sign of $a(x)-a(y)$ for all $a \in B_2,$ and $k_{x,y}=-\log_2 \dc(x,y).$
Let $x,y\in \mathcal{C}_2, x\neq y.$ First, assume that $k_{x,y}=0.$  Thus $x_0\neq y_0.$
\begin{enumerate}
    \item If $x_0=0, y_0=1$, then  $u_0(x)=2\eta_0(x)-1=2\cdot x_0-1=-1,u_0(y)=2\eta_0(y)-1=2\cdot y_0-1=1$, and hence $u_0(x)-u_0(y)=-2$.
    \begin{enumerate}
        \item If $x_1=0,y_1=0$, then $u_1(x)-u_1(y)=1-1=0$ and $u_0(x)u_1(x)-u_0(y)u_1(y)=-1\cdot (-1)-1\cdot(- 1)=2.$ Thus, we have
        \[
        2^{k_{x,y}+1}\left|\sum_{a \in \sigma_{x,y}} \pm_{a,x,y} \alpha_a\right|=2|-\alpha_0+\alpha_2|=2|\alpha_0-\alpha_2|.
        \]
        \item If $x_1=0, y_1=1,$ then $u_1(x)-u_1(y)=-1-1=-2$ and $ u_0(x)u_1(x)-u_0(y)u_1(y)=-1\cdot (-1)-1\cdot 1=1-1=0.$ Thus, we have
         \[
        2^{k_{x,y}+1}\left|\sum_{a \in \sigma_{x,y}} \pm_{a,x,y} \alpha_a\right|=2|-\alpha_0-\alpha_1|=2|\alpha_0+\alpha_1|.
        \]
        \item If $x_1=1, y_1=0,$ then $u_1(x)-u_1(y)=1-(-1)=2$ and $ u_0(x)u_1(x)-u_0(y)u_1(y)=-1\cdot 1-1\cdot (-1)=-1+1=0.$ Thus, we have
         \[
        2^{k_{x,y}+1}\left|\sum_{a \in \sigma_{x,y}} \pm_{a,x,y} \alpha_a\right|=2|-\alpha_0+\alpha_1|=2|\alpha_0-\alpha_1|.
        \]
         \item If $x_1=1, y_1=1,$ then $u_1(x)-u_1(y)=1-(1)=0$ and $ u_0(x)u_1(x)-u_0(y)u_1(y)=-1\cdot 1-1\cdot (1)=-1-1=-2.$ Thus, we have
         \[
        2^{k_{x,y}+1}\left|\sum_{a \in \sigma_{x,y}} \pm_{a,x,y} \alpha_a\right|=2|-\alpha_0-\alpha_2|=2|\alpha_0+\alpha_2|.
        \]
    \end{enumerate}
    \item If $x_0=1,y_0=0$, then $u_0(x)=1,$ $u_0(y)=-1$, and $u_0(x)-u_0(y)=1-(-1)=2$.
     \begin{enumerate}
        \item If $x_1=0,y_1=0$, then $u_1(x)-u_1(y)=-1-(-1)=0$ and $u_0(x)u_1(x)-u_0(y)u_1(y)=1\cdot (-1) 1-(-1)\cdot (-1)=-2.$ Thus, we have
        \[
        2^{k_{x,y}+1}\left|\sum_{a \in \sigma_{x,y}} \pm_{a,x,y} \alpha_a\right|=2|\alpha_0-\alpha_2|.
        \]
        \item If $x_1=0, y_1=1,$ then $u_1(x)-u_1(y)=-1-1=-2$ and $ u_0(x)u_1(x)-u_0(y)u_1(y)=1\cdot (-1)-(-1)\cdot 1=-1+1=0.$ Thus, we have
         \[
        2^{k_{x,y}+1}\left|\sum_{a \in \sigma_{x,y}} \pm_{a,x,y} \alpha_a\right|=2|\alpha_0-\alpha_1|.
        \]
        \item If $x_1=1, y_1=0,$ then $u_1(x)-u_1(y)=1-(-1)=2$ and $ u_0(x)u_1(x)-u_0(y)u_1(y)=1\cdot 1-(-1)\cdot (-1)=1-1=0.$ Thus, we have
         \[
        2^{k_{x,y}+1}\left|\sum_{a \in \sigma_{x,y}} \pm_{a,x,y} \alpha_a\right|=2|\alpha_0+\alpha_1|.
        \]
         \item If $x_1=1, y_1=1,$ then $u_1(x)-u_1(y)=1-(1)=0$ and $ u_0(x)u_1(x)-u_0(y)u_1(y)=1\cdot 1-(-1)\cdot (1)=1+1=2.$ Thus, we have
         \[
        2^{k_{x,y}+1}\left|\sum_{a \in \sigma_{x,y}} \pm_{a,x,y} \alpha_a\right|=2|\alpha_0+\alpha_2|.
        \]
    \end{enumerate}
\end{enumerate}
Second, assume that $k_{x,y}=1$, then $x_0=y_0$ and thus $u_0(x)-u_0(y)=0,$ and $x_1\neq y_1.$
\begin{enumerate}
    \item If $x_0=y_0=0$, then $u_0(x)=-1$ and $u_0(y)=-1$ and
    \begin{enumerate}
        \item If $x_1=0,y_1=1$, then $u_1(x)-u_1(y)=-1-1=-2$, and $ u_0(x)u_1(x)-u_0(y)u_1(y)=-1\cdot(-1)-(-1)\cdot 1=1+1=2. $ Thus, we have
         \[
        2^{k_{x,y}+1}\left|\sum_{a \in \sigma_{x,y}} \pm_{a,x,y} \alpha_a\right|=4|-\alpha_1+\alpha_2|=4|\alpha_1-\alpha_2|.
        \]
         \item If $x_1=1,y_1=0$, then $u_1(x)-u_1(y)=1-(-1)=2$, and $ u_0(x)u_1(x)-u_0(y)u_1(y)=-1\cdot(1)-(-1)\cdot (-1)=-1-1=-2. $ Thus, we have
         \[
        2^{k_{x,y}+1}\left|\sum_{a \in \sigma_{x,y}} \pm_{a,x,y} \alpha_a\right|=4|\alpha_1-\alpha_2|.
        \]
    \end{enumerate}
    \item If $x_0=y_0=1,$ then $u_0(x)=1$ and $u_0(y)=1$ and
     \begin{enumerate}
        \item If $x_1=0,y_1=1$, then $u_1(x)-u_1(y)=-1-1=-2$, and $ u_0(x)u_1(x)-u_0(y)u_1(y)=1\cdot(-1)-(1)\cdot 1=-1-1=-2. $ Thus, we have
         \[
        2^{k_{x,y}+1}\left|\sum_{a \in \sigma_{x,y}} \pm_{a,x,y} \alpha_a\right|=4|-\alpha_1-\alpha_2|=4|\alpha_1+\alpha_2|.
        \]
         \item If $x_1=1,y_1=0$, then $u_1(x)-u_1(y)=1-(-1)=2$, and $ u_0(x)u_1(x)-u_0(y)u_1(y)=1\cdot(1)-(1)\cdot (-1)=1+1=2. $ Thus, we have
         \[
        2^{k_{x,y}+1}\left|\sum_{a \in \sigma_{x,y}} \pm_{a,x,y} \alpha_a\right|=4|\alpha_1+\alpha_2|.
        \]
    \end{enumerate}
\end{enumerate}
Thus, all cases are finished since $k_{x,y} \leq 1$, and the proof is complete.
\end{proof}
Next, we calculate $\lac$ on $A_2.$
\begin{theorem}\label{t:lac2-formula}
Let $f \in A_2$ such that $f=\alpha_0 u_0+\alpha_1 u_1+\alpha_2 u_0 u_1$ for some $\alpha_0, \alpha_1, \alpha_2 \in \C$.

It holds that
\[
\lac(f)=\max\left\{\begin{array}{l}
2|\alpha_0+\alpha_1-\alpha_2|,\ 2|\alpha_0-\alpha_1+\alpha_2|, \ 2|\alpha_0-\alpha_1-\alpha_2|,\\
2|\alpha_0+\alpha_1+\alpha_2|, \ 4|\alpha_1-\alpha_2|, \ 4|\alpha_1+\alpha_2|
\end{array} \right\}.
\]
\end{theorem}
\begin{proof}
Note $B_2=\{u_0, u_1, u_0u_1\}$  by Notation \ref{n:dense-subalg}. By Theorem \ref{t:lac-formula}, we have
\[
\lac(f)=\max_{k \in \{0,1\}} \left\{ \max_{x \in \mathcal{C}_2} \left\{2^{k +1}\left|\sum_{a \in \rho_k} \pm_{a,x} \alpha_a\right| \right\}\right\},
\]
where we set $\alpha_{u_0}=\alpha_0, \alpha_{u_1}=\alpha_1, \alpha_{u_0u_1}=\alpha_2,  $ and $\rho_k=\{a \in B_2 \mid E_k(a)=0_{C(\mathcal{C})}\}$ for all $k \in \{0,1\}$ and $\pm_{a,x}$ is the sign of $a(x)$ for all $a \in B_2,x \in \mathcal{C}_2.$

First, let $k=0.$ By Lemma \ref{l:en-vanish} and (7) of Lemma \ref{l:en}, we have that $\rho_0=\{u_0, u_1, u_0u_1 \}.$ Let $x \in \mathcal{C}_2.$
\begin{enumerate}
    \item If $x_0=x_1=0$, then $u_0(x)=-1, u_1(x)=-1, u_0(x)u_1(x)=1$, and thus 
    \[
     2^{k+1}\left| \sum_{a \in \rho_k}\pm_{a,x} \alpha_a  \right| =2|-\alpha_0-\alpha_1+\alpha_2|=2|\alpha_0+\alpha_1-\alpha_2|.
    \]
      \item If $x_0=0,x_1=1$, then $u_0(x)=-1, u_1(x)=1, u_0(x)u_1(x)=-1$, and thus 
    \[
     2^{k+1}\left| \sum_{a \in \rho_k}\pm_{a,x} \alpha_a  \right| =2|-\alpha_0+\alpha_1-\alpha_2|=2|\alpha_0-\alpha_1+\alpha_2|.
    \]
     \item If $x_0=1,x_1=0$, then $u_0(x)=1, u_1(x)=-1, u_0(x)u_1(x)=-1$, and thus 
    \[
     2^{k+1}\left| \sum_{a \in \rho_k}\pm_{a,x} \alpha_a  \right| =2|\alpha_0-\alpha_1-\alpha_2|.
    \]
     \item If $x_0=1,x_1=1$, then $u_0(x)=1, u_1(x)=1, u_0(x)u_1(x)=1$, and thus 
    \[
     2^{k+1}\left| \sum_{a \in \rho_k}\pm_{a,x} \alpha_a  \right| =2|\alpha_0+\alpha_1+\alpha_2|.
    \]
\end{enumerate}
Second, let $k=1.$ By Lemma \ref{l:en-vanish} and (7) of Lemma \ref{l:en}, we have $\rho_0=\{ u_1, u_0u_1 \}.$ Let $x \in \mathcal{C}_2.$
\begin{enumerate}
     \item If $x_0=x_1=0$, then $u_0(x)=-1, u_1(x)=-1, u_0(x)u_1(x)=1$, and thus 
    \[
     2^{k+1}\left| \sum_{a \in \rho_k}\pm_{a,x} \alpha_a  \right| =4| -\alpha_1+\alpha_2|=4| \alpha_1-\alpha_2|.
    \]
      \item If $x_0=0,x_1=1$, then $u_0(x)=-1, u_1(x)=1, u_0(x)u_1(x)=-1$, and thus 
    \[
     2^{k+1}\left| \sum_{a \in \rho_k}\pm_{a,x} \alpha_a  \right| =4| \alpha_1-\alpha_2|.
    \]
     \item If $x_0=1,x_1=0$, then $u_0(x)=1, u_1(x)=-1, u_0(x)u_1(x)=-1$, and thus 
    \[
     2^{k+1}\left| \sum_{a \in \rho_k}\pm_{a,x} \alpha_a  \right| =4|-\alpha_1-\alpha_2|=4|\alpha_1+\alpha_2|.
    \]
     \item If $x_0=1,x_1=1$, then $u_0(x)=1, u_1(x)=1, u_0(x)u_1(x)=1$, and thus 
    \[
     2^{k+1}\left| \sum_{a \in \rho_k}\pm_{a,x} \alpha_a  \right| =4| \alpha_1+\alpha_2|,
    \]
\end{enumerate}
which completes the proof.
\end{proof}
Therefore, the $A_2$ case displays how $\lac$ seems to be more sensitive to the coefficients by allowing one to vary them to impact the entire quantity rather than only being able to compare coefficients pairwise in $\ldc.$  Let's now use these formulas to find many elements in $A_2$ that separate $\lac$ and $\ldc.$
\begin{theorem}\label{t:lip-sep}
If $f=\alpha_0u_0+\alpha_1u_1+\alpha_2 u_0 u_1$ such that $\alpha_0, \alpha_1, \alpha_2 \in \R$ with 
\[
\alpha_0>\alpha_1+\alpha_2 \text{ and } \alpha_1=\alpha_2>0,
\]
then
\[
\ldc(f)=\max\{ 2(\alpha_0 + \alpha_1), 8 \alpha_1\}
\]
and 
\[
\lac(f)=2(\alpha_0+2\alpha_1)=2(\alpha_0+2\alpha_2)
\]
and $\ldc(f)<\lac(f).$
In particular, $\ldc(4u_0+u_1+u_0u_1)=10<12=\lac(4u_0+u_1+u_0u_1).$
\end{theorem}
\begin{proof}
First, we consider $\ldc$. First, $\alpha_0-\alpha_2>\alpha_1>0$, so $2|\alpha_0-\alpha_2|=2(\alpha_0-\alpha_2)< 2(\alpha_0+\alpha_2)=2(\alpha_0+\alpha_1)$ since $\alpha_2>0.$ Similarly, $2|\alpha_0-\alpha_1|=2(\alpha_0-\alpha_1)< 2(\alpha_0+\alpha_1).$ Also, $2|\alpha_0+\alpha_1|=2(\alpha_0+\alpha_1)$ and $2|\alpha_0+\alpha_2|=2(\alpha_0+\alpha_2)=2(\alpha_0+\alpha_1)$.  Now, $4|\alpha_1-\alpha_2|=0$ and $4|\alpha_1+\alpha_2|=8\alpha_1.$ This proves the formula for $\ldc(f).$

Second, we consider $\lac$. Now, $\alpha_0+\alpha_1-\alpha_2>\alpha_1+\alpha_2+\alpha_1-\alpha_2=2\alpha_1>0$.  Thus $2|\alpha_0+\alpha_1-\alpha_2|=2(\alpha_0+\alpha_1-\alpha_2)< 2(\alpha_0+\alpha_1+\alpha_2)$ since $\alpha_2>0.$ Similarly $2|\alpha_0-\alpha_1+\alpha_2|=2(\alpha_0-\alpha_1+\alpha_2)< 2(\alpha_0+\alpha_1+\alpha_2)$ and $2|\alpha_0-\alpha_1-\alpha_2|=2(\alpha_0-\alpha_1-\alpha_2)< 2(\alpha_0+\alpha_1+\alpha_2).$ However  $2(\alpha_0+\alpha_1+\alpha_2)=2(\alpha_0+2\alpha_1)=2(\alpha_0+2\alpha_2).$ Now $4|\alpha_1-\alpha_2|=0$ and $4|\alpha_1+\alpha_2|=8\alpha_1.$ But $2(\alpha_0+\alpha_1+\alpha_2)>2(\alpha_1+\alpha_2+\alpha_1+\alpha_2)=2(4\alpha_1)=8\alpha_1, $
which proves the formula for $\lac(f).$ 

Next, $2(\alpha_0+\alpha_1)<2(\alpha_0+2\alpha_1)$ since $\alpha_1>0.$ And, we already showed that $8\alpha_1<2(\alpha_0+2\alpha_1),$ which establishes $\ldc(f)<\lac(f)$.  Lastly, $\alpha_0=4, \alpha_1=1, \alpha_2=1$ satisfy the assumption and completes the proof. 
\end{proof}

Although we showed that $\ldc$ and $\lac$ separate on $A_2$, the fact that $A_2$ is finite-dimensional still allows us to compare $\ldc$ and $\lac$. Indeed, since $\ldc$ and $\lac$ vanish on the same subspace and $A_2$ is finite-dimensionl, we have that  $\ldc$ and $\lac$ are equivalent since all norms on finite-dimensional vector spaces are equivalent (see \cite[Theorem III.3.1]{Conway90}) and $\ldc$ and $\lac$ are norms on the quotient space (which is still finite-dimensional) given by the subspace they vanish on. However, this result is about existence and doesn't provide a way to find explicit constants for equivalence. So, in Theorem \ref{t:a2-equiv}, we find such constants.  First, we present some basic inequaliites. 
\begin{lemma}\label{l:frd-rk-sn}
If $x,y \in \C$, then 
\[
|x|\leq \max\{ |x+y|, \ |x-y|\}.
\]
\end{lemma}
\begin{proof}
We have
\begin{align*}
    2|x|&=|2x|=|x+x|=|x+y-y+x|\leq |x+y|+|x-y|\\
    &\leq 2\max\{ |x+y|, \ |x-y|\},
\end{align*}
which implies $|x| \leq \max\{ |x+y|, \ |x-y|\}.$
\end{proof}
\begin{lemma}\label{l:frd-rx-n-l"t"}
If $x,y,z \in \C,$ then 
\[
|x+y-z|\leq  \max\{2|x+y|, \ 2|x+z|, \ 2|x-z|\}.
\]
\end{lemma}
\begin{proof}
We have
\begin{align*}
    |x+y-z|&\leq |x+y|+|z|  \leq 2\max\{|x+y|,\ |z|\}\\
    & \leq 2\max\{|x+y|, \ \max\{|x+z|, \ |x-z|\}\}\\
    & = 2\max\{|x+y|, \ |x+z|, \ |x-z|\},
\end{align*}
where the second to last line is provided by Lemma \ref{l:frd-rk-sn}. 
\end{proof}

\begin{theorem}\label{t:a2-equiv}
If $f \in A_2,$ then 
\[
\ldc(f) \leq \lac(f) \leq 2\ldc(f).
\]
\end{theorem}
\begin{proof}
We begin with the first inequality.  Since $\ldc(1_{C(\mathcal{C})})=\lac(1_{C(\mathcal{C})})=0$ since they are Lip-norms, we only need to consider $f \in A_2$ such that $f=\alpha_0u_0+\alpha_1u_1+\alpha_2u_0u_1$ for some $\alpha_0, \alpha_1, \alpha_2 \in \C$ by the argument of Corollary \ref{c:l-agree}.

By Lemma \ref{l:frd-rk-sn} and Theorem \ref{t:lac2-formula}, we have
\[
2|\alpha_0-\alpha_2| \leq \max\{ 2|\alpha_0-\alpha_1-\alpha_2|, \ 2|\alpha_0+\alpha_1-\alpha_2|\}\leq \lac(f),
\]
\[
2|\alpha_0-\alpha_1| \leq \max\{ 2|\alpha_0-\alpha_1-\alpha_2|, \ 2|\alpha_0-\alpha_1+\alpha_2|\}\leq \lac(f),
\]
\[
2|\alpha_0+\alpha_1| \leq \max\{ 2|\alpha_0+\alpha_1-\alpha_2|, \ 2|\alpha_0+\alpha_1+\alpha_2|\}\leq \lac(f),
\]
\[
2|\alpha_0+\alpha_2| \leq \max\{ 2|\alpha_0+\alpha_1+\alpha_2|, \ 2|\alpha_0-\alpha_1+\alpha_2|\}\leq \lac(f),
\]
and 
\[ 4|\alpha_0-\alpha_2|\leq \lac(f) \ \text{ and } \ 4|\alpha_0+\alpha_2|\leq \lac(f). \]
Hence, by Theorem \ref{t:ldc2-formula}, each term defining $\ldc(f)$ is less than or equal to $\lac(f),$ and thus the maximum, $\ldc(f)$, is less than or equal to $\lac(f)$.  Therefore, 
\[
\ldc(f)\leq \lac(f).
\]

Next, we consider the second inequality. By Lemma \ref{l:frd-rx-n-l"t"} and Theorem \ref{t:ldc2-formula}, we have
\begin{align*}
2|\alpha_0+\alpha_1-\alpha_2|&\leq 2\max\{2|\alpha_0+\alpha_1|, \ 2|\alpha_0+\alpha_2|, \ 2|\alpha_0-\alpha_2| \}\\
& \leq 2\max\{\ldc(f), \ \ldc(f), \ \ldc(f) \}=2\ldc(f),
\end{align*}
\[
2|\alpha_0-\alpha_1+\alpha_2|\leq 2\max\{2|\alpha_0-\alpha_1|, \ 2|\alpha_0+\alpha_2|, \ 2|\alpha_0-\alpha_2| \} \leq  2\ldc(f),
\]
\[
2|\alpha_0-\alpha_1-\alpha_2|\leq 2\max\{2|\alpha_0-\alpha_1|, \ 2|\alpha_0-\alpha_2|, \ 2|\alpha_0+\alpha_2| \} \leq  2\ldc(f),
\]
\[
2|\alpha_0+\alpha_1+\alpha_2|\leq 2\max\{2|\alpha_0+\alpha_1|, \ 2|\alpha_0+\alpha_2|, \ 2|\alpha_0-\alpha_2| \} \leq  2\ldc(f),
\]
and 
\[ 4|\alpha_0-\alpha_2|\leq \ldc(f) \ \text{ and } \ 4|\alpha_0+\alpha_2|\leq \ldc(f). \]
Hence, by Theorem \ref{t:ldc2-formula}, each term defining $\lac(f)$ is less than or equal to $2\ldc(f),$ and thus the maximum, $\lac(f)$, is less than or equal to $2\ldc(f)$.  Therefore, 
\[
\lac(f)\leq 2\ldc(f),
\]
which completes the proof.
\end{proof}

Thus, we have successfully separated $\lac$ and $\ldc$, while also discovering some interesting structural differences and similarities between the two. One route to go next would be to see if we can continue finding equivalence constants for higher dimensional spaces than $A_2$. We note that we are not even certain if we have found the tightest constant on the right inequality in Theorem \ref{t:a2-equiv}. It may be that the tighter number is less than $2$. Another route is to compare the domains  $\dom{\lac}$ and $\dom{\ldc}.$  Yes, Theorem \ref{t:lac} and Theorem \ref{t:domain} show that $\cup_{n \in \N} A_n \subseteq \dom{\lac}\cap\dom{\ldc}$, but this doesn't mean the domains are necessarily equal. Our formulas for $\lac$ and $\ldc$ (Theorem \ref{t:lac-formula} and Theorem \ref{t:ldc-formula}, respectively) may be the key to figuring this out. The formula for $\lac$ will continue to consider more and more coefficients as the dimension approaches infinity in comparison to $\ldc$. Thus, it may be the case that value approaches infinity while the other stays bounded, which would establish difference in domains.
\bibliographystyle{plain}
\bibliography{thesis-a}

\end{document}